\documentclass[10pt]{amsart}
\usepackage{amssymb}

\newtheorem{theorem}{Theorem}

\newtheorem{corollary}[theorem]{Corollary}

\newtheorem{lemma}[theorem]{Lemma}

\newcommand{\R}{{\mathbb R}}

\begin{document}

{

\title{Inscribed squares of level sets of functions on the sphere}

\author{V.A.~Vassiliev}
\address{Weizmann Institute of Science, Rehovot, Israel}
 \email{vavassiliev@gmail.com}

\begin{abstract}
According to a theorem by F.~Dyson, every continuous function $f: S^2 \to {\mathbb R}$ takes the same value at the vertices of a square inscribed into a great circle of $S^2$. We give a proof of this statement based on the theory of characteristic classes and indicate other potential applications of this approach.

\end{abstract} 
\thanks{}
\maketitle

Consider the unit sphere $S^2 \subset {\mathbb R}^3 $ defined by the equation 
\begin{equation}
\label{sphere}
x^2 + y^2 + z^2 =1. 
\end{equation}
The {\em great circles} are its intersections with the two-dimensional subspaces in ${\mathbb R}^3$.

\begin{theorem}[see \cite{d}]
\label{mthm}
For any continuous function $f: S^2 \to {\mathbb R}$, there exists a square in ${\mathbb R}^3$ inscribed into a great circle of $S^2$, such that $f$ takes equal values at all its vertices.
\end{theorem}

An alternative proof of this theorem occupies almost the entire rest of the article.
\medskip

 Consider the configuration  space $B(S^2, 4)$, whose points are subsets of cardinality four in $S^2$. The {\em regular bundle} $E$ on $B(S^2, 4)$ is the four-dimensional vector bundle whose fiber over a set of cardinality four is the space of all real-valued functions on that set. This bundle is reducible: it contains a one-dimensional trivial subbundle whose fibers are the constant functions. Denote the three-dimensional quotient bundle of $E$ through this one-dimensional subbundle   by $\tilde E$. 

The mod 2 Euler class of this quotient bundle, $w_3(\tilde E) \equiv w_3(E) \in H^3(B(S^2,4), {\mathbb Z}_2)$,  is non-trivial: see, for example, \cite{fuks}, where this is proved for the restriction of the bundle $E$ (and hence also of $\tilde E$) to the subspace  $B({\mathbb R}^2, 4) \subset B(S^2, 4)$.

For each continuous function $f: S^2 \to {\mathbb R},$ the corresponding {\em evaluation cross-section} of the bundle $E$ is defined by the restrictions of $f$ to the four-configurations.  Therefore, a continuous cross-section of the quotient bundle $\tilde E$ is also defined by factoring this cross-section through the trivial subbundle. The intersection points of this cross-section with the zero section of the bundle $\tilde E$ are exactly the quadruples of points at which the function $f$ has the same value.

\begin{lemma}
\label{lem0}
Let $U$ be a subset of $B(S^2,4)$. If there exists a function $f: S^2 \to {\mathbb R}$ that 
takes at least two different values at the points of any four-configuration from $ U$, then the restriction of the class $w_3(\tilde E)$ to the space $U$ is the zero element of the group $H^3(U, {\mathbb Z}_2)$.
\end{lemma}  

\noindent
{\it Proof.} In this case, the evaluation cross-section defined by $f$ is non-zero everywhere on $U$. Our statement now follows from  Proposition 4 in Section 4 of \cite{MS}. \hfill $\Box$
\medskip

Consider the manifold $M^3$ of all squares inscribed into great circles of $S^2$. It is a compact orientable three-dimensional manifold with the obvious structure of a fiber bundle over the space ${\mathbb R}P^2$ of all great circles, with the fiber $S^1/{\mathbb Z}_4 \simeq S^1$. 
In particular, $H_3(M^3, {\mathbb Z}_2) \simeq {\mathbb Z}_2$.

Consider the restriction of the function 
\begin{equation}
\label{fun0}
f_0 = z - x y
\end{equation}
to $S^2$. There are two connected submanifolds in the eleven-dimensional total space of the vector bundle $\tilde E$: the three-dimensional manifold, which is the image of the evaluation cross-section on $M^3$ defined by the function (\ref{fun0}), and the eight-dimensional  zero section of the entire vector bundle. The former submanifold is compact, and the latter submanifold is closed in the total space of the bundle $\tilde E$, in particular it defines a Borel--Moore homology class of this space.  Therefore, the mod 2 intersection index of these two manifolds is well-defined.

\begin{lemma}
\label{lem2}
The mod 2 intersection index of these two manifolds is non-zero.
\end{lemma}

\noindent
{\it Proof.} 
We will identify the zero section of the vector bundle $\tilde E$ with the base space $B(S^2,4)$. 
The intersection points in question are the sets of vertices of inscribed squares, such that all of these vertices simultaneously satisfy an equation of the form $z -x y =c$ for some constant $c$. As each inscribed square is centrally symmetric, the equation $(-z) - (-x)(-y) =c$ is also satisfied by these vertices, and hence the square lies in the plane $\{z=0\}$. The set of points of this plane that satisfy the equations $x y = c$ and $x^2 + y^2 =1$ contains a square if and only if $c=0$. The coordinates of the vertices of the square are then equal to 
\begin{equation}
\label{co3}
(1, 0), \ (0, 1), \ (-1, 0), \ \mbox{and} \ (0,-1). 
\end{equation} 
It remains to be proven that this four-configuration contributes 1 to the intersection index. 

Each two-subspace of ${\mathbb R}^3$ near the plane $\{z=0\}$ has the equation $z = \alpha x + \beta y$. Thus, for the local coordinates of the manifold $M^3$ near the
 configuration (\ref{co3}) we can take the coefficients $\alpha$ and $\beta$ defining the circle, and also the coordinate $x$ of the point of the ${\mathbb Z}_4$-invariant four-configuration in this circle, which is close to the point $(1,0)$ of the four-configuration (\ref{co3}). The partial derivatives of the values of $f_0$ at the four points (\ref{co3}) with respect to  these three local coordinates $\alpha, \beta,$ and $x$ form the upper three rows of the matrix
$$\begin{pmatrix}
1 & 0 & -1 & 0 \\
0 & 1 & 0 & -1 \\
-1 & 1 & -1 & 1 \\
1 & 1 & 1 & 1
\end{pmatrix} .
$$
This matrix is non-degenerate, therefore the evaluation cross-section of the bundle $\tilde E$ over $M^3$ defined by the function (\ref{fun0})  intersects the zero section of the bundle $\tilde E$  transversally at the point defined by the configuration (\ref{co3}). \hfill $\Box$

\begin{corollary}
\label{cor0}
The class $w_3(\tilde E)$ takes the non-zero value on the fundamental class of the manifold $M^3$.  
\end{corollary}

\noindent
{\it Proof.} Let $\tilde E_0$ be the space of the bundle $\tilde E$ less the zero section. According to Lemma \ref{lem2}, the homology class of the manifold $M^3$ in the group $H_3(\tilde E, \tilde E_0; {\mathbb Z}_2)$ coincides with the homology class of an arbitrary fiber of the bundle $\tilde E$. The ``orientation'' class $u \in H^3(\tilde E, \tilde E_0; {\mathbb Z}_2)$ (see \cite{MS}, \S 9) takes the non-zero value on the class of the fiber, and hence also on $M^3$. By the definition of the mod 2 Euler class (see {\it ibid}), this value is equal to the value of the class $w_3(\tilde E)$ on $M^3$. \hfill $\Box$
\medskip

Theorem \ref{mthm} follows immediately from comparing Lemma \ref{lem0} and Corollary \ref{cor0}. \hfill $\Box$

\medskip
\noindent
{\bf Similar problems.}
Is it true that any continuous map $S^5 \to \R^3$ takes equal values at some five vertices of an octahedron inscribed into a great two-dimensional subsphere of $S^5$?

Is it true that  any continuous map $S^5 \to \R^2$ takes equal values at all vertices of a hexadecachoron inscribed into a great three-dimensional subsphere of $S^5$?

\medskip
I thank Chris Gerig very much for drawing my attention to the work of F.~Dyson.

}


\begin{thebibliography}{9}

\bibitem{d} F.J.~Dyson, {\it Continuous functions defined on spheres. Annals of Mathematics} 54(3), 1951, 534-536.

\bibitem{fuks} D.B.~Fuchs, {\it Cohomologies of the braid group mod 2,} Funct. Anal. Appl., 4:2 (1970), 143—151.

\bibitem{MS}  J.W.~Milnor, J.D.~Stasheff; {\it Characteristic classes.}
Princeton Univ. Press and Univ. of Tokyo Press. Princeton, New Jersey, 1974.



\end{thebibliography}
\end{document}